\documentclass[12pt]{article}
\usepackage{amsmath}
\usepackage{float}
\usepackage{amsthm}
\usepackage{amssymb}
\usepackage[utf8x]{inputenc}
\usepackage{caption}
\usepackage{graphics}
\usepackage{enumerate} 
\usepackage{algorithmic}
\usepackage{calrsfs}
\DeclareMathAlphabet{\pazocal}{OMS}{zplm}{m}{n}

\usepackage[Algorithm,ruled]{algorithm}
\usepackage{titlesec}
\usepackage{sectsty}
\titleformat*{\section}{\LARGE\bfseries}
\titleformat*{\subsection}{\Large\bfseries}
\titleformat*{\subsubsection}{\large\bfseries}
\sectionfont{\Huge}
\subsectionfont{\large}
\subsubsectionfont{\normalsize}
\usepackage{multicol}
\usepackage{lmodern}
\usepackage{marvosym} 
\usepackage{lipsum}
\usepackage{mwe}
\usepackage{caption}
\usepackage{subcaption} 
\newtheoremstyle{case}{}{}{}{}{}{:}{ }{}
\theoremstyle{case}

\newcommand{\be}{\begin{equation}}
\newcommand{\ee}{\end{equation}}
\newcommand{\ben}{\begin{eqnarray*}}
\newcommand{\een}{\end{eqnarray*}}
\newtheorem{examp}{\sc example}
\newtheorem{remk}{\sc remark}
\newtheorem{corol}{\sc corollary}
\newtheorem{lemma}{\sc lemma}
\newtheorem{theorem}{\sc theorem}
\newtheorem{defn}{\sc definition}
\newtheorem{prop}{\sc proposition}
\newcommand{\bt}{\begin{theorem}}
\newcommand{\et}{\end{theorem}}
\newcommand{\bl}{\begin{lemma}}
\newcommand{\el}{\end{lemma}}
\newcommand{\bed}{\begin{defn}}
\newcommand{\eed}{\end{defn}}
\newcommand{\brem}{\begin{remk}}
\newcommand{\erem}{\end{remk}}
\newcommand{\bex}{\begin{examp}}
\newcommand{\eex}{\end{examp}}
\newcommand{\bcl}{\begin{corol}}
\newcommand{\ecl}{\end{corol}}

\newcommand{\NI}{\noindent}

\theoremstyle{definition}
\theoremstyle{remark}

\numberwithin{equation}{section}
\numberwithin{theorem}{section}
\numberwithin{lemma}{section}

\begin{document}

\title{\large\bf\sc Complementarity Problem With Nekrasov $Z$ tensor}

\author{ 
R. Deb$^{a,1}$ and A. K. Das$^{b,2}$\\
\emph{\small $^{a}$Jadavpur University, Kolkata, India.}\\	
\emph{\small $^{b}$Indian Statistical Institute, Kolkata, India.}\\
\emph{\small $^{1}$Email: rony.knc.ju@gmail.com}\\
\emph{\small $^{2}$Email: akdas@isical.ac.in}\\
}

\date{}

\maketitle

\begin{abstract}
\NI It is worth knowing that a particular tensor class belongs to $P$-tensor which ensures the compactness to solve tensor complementarity problem (TCP). In this study, we propose a new class of tensor, Nekrasov $Z$ tensor, in the context of the tensor complementarity problem. We show that the class of $P$-tensor contains the class of even ordered Nekrasov $Z$ tensors with positive diagonal elements. In this context, we propose a procedure by which a Nekrasov $Z$ tensor can be transformed into a tensor which is diagonally dominant.\\


\end{abstract}
\footnotetext[1]{Corresponding author}

\section{Introduction}
$\mathcal{T}= (t_{j_1 ... j_m}) $ has multidimensional array of elements $t_{j_1... j_m} \in \mathbb{C}$ where $j_i \in [n]$ with $i\in [m],$ $[n]=\{1,...,n\}.$ The set of $m$th order $n$ dimensional complex (real) tensors is denoted by $\mathbb{C}^{[m,n]}\; (\mathbb{R}^{[m,n]}).$ 
Several classes of tensors have significant impact on different branches of science and engineering. Applications of tensors can be found in physics, quantum computing, diffusion tensor imaging, image authenticity verification problem, spectral hypergraph theory, optimization theory and in several other areas. In optimization theory, Song and Qi \cite{song2017properties} introduced the tensor complementarity problem (TCP) which is identified a subclass of the nonlinear complementarity problem. Here functions are obtained with the help of a tensor.

\NI Given $\mathcal{T}\in \mathbb{R}^{[m,n]}$ and $q\in\mathbb{R}^n,$ TCP is to obtain $x\in \mathbb{R}^n $ such that
\begin{equation}\label{tensor comp equation}
	x\geq 0,\;\;  \mathcal{T}x^{m-1}+q \geq 0,\;\;  x^T (\mathcal{T}x^{m-1}+q)  = 0.
\end{equation}
This problem is denoted by TCP$(\mathcal{T}, q)$ and the solution set of TCP$(\mathcal{T}, q)$ is denoted by SOL$(\mathcal{T}, q).$
TCP are observed in optimization theory, game theory and in other areas. Some mathematical problem can be reformulated TCP, such as a class of multiperson noncooperative game \cite{huang2017formulating}, hypergraph clustering problem and traffic equilibrium problem \cite{huang2019tensor}.

\NI TCP can be viewed as one kind of extension of the linear complementarity problem (LCP) where the involved functions are homogeneous polynomials constructed with the help of a tensor.
\NI Given $M\in \mathbb{R}^{[2,n]}$ and $q\in \mathbb{R}^n$, the LCP \cite{cottle2009linear} is to obtain $u\in \mathbb{R}^n$ such that
\begin{equation}\label{linear comp equation}
 u\geq 0,\;\;  Mu + q \geq 0,\;\; u^T (Mu + q) = 0.
\end{equation}
The problem is denoted by LCP$(M,q)$ and the solution set of LCP$(M,q)$ is denoted by SOL$(M,q).$ The stability of a large number of problems can be explained with the formulation of a linear complementarity problem. For details see \cite{das2019some}, \cite{neogy2006some}, \cite{dutta2021some} \cite{dutta2022on}, \cite{neogy2013weak}, \cite{jana2018semimonotone}, \cite{neogy2005almost}, \cite{neogy2011singular}, \cite{jana2019hidden}, \cite{jana2021more}, \cite{neogy2009modeling}, \cite{das2017finiteness}, \cite{das2018invex}. For the formulation of conflicting situations see \cite{mondal2016discounted}, \cite{neogy2008mathematical}, \cite{neogy2008mixture}, \cite{neogy2005linear}, \cite{neogy2016optimization}, \cite{das2016generalized}, \cite{neogy2011generalized} and  \cite{mohan2004note}. In connection with linear complementarity problem computation methods are dependent on the matrix classes. For details see \cite{mohan2001classes}, \cite{mohan2001more} \cite{neogy2005principal}, \cite{das2016properties}, \cite{neogy2012generalized}, \cite{jana2019hidden}, \cite{jana2021iterative}, \cite{jana2021more}, \cite{jana2018processability}.

\NI Nekrasov $Z$-matrix was introduced by Orea and Pe$\tilde{n}$a \cite{orera2019accurate} to study the methods to compute its inverse with high relative accuracy.

\NI Several matrix types play a crusial role in the theory of linear complementarity problem, similarly, some structured tensors are important to study tensor complementarity theory. In recent years some special types of matrices are extended to tensors in connection with TCP. Among the several tensor classes, the class of $P(P_0)$-tensor, $Z$ tensor, $M$-tensor, $H$ tensor, and Nekrasov tensor are of special interest. $Z$ tensor and $M$-tensor are introduced by Zhang et al. \cite{zhang2014m} in context of tensor complementarity problem. They prove that for a $Z$ tensor the TCP has the least element property. Ding et al. \cite{ding2013m} introduce $H$ tensor and establish its importance in the tensor complementarity problem. Zhang and Bu \cite{zhang2018nekrasov} introduce the Nekrasov tensor and show that a Nekrasov tensor can be posed as a $H$ tensor.

\NI Song and Qi \cite{song2014properties} introduce $P$-tensor. Various properties of $P$-tensors are considered in connection with TCP. For details see \cite{bai2016global}, \cite{ding2018p}, \cite{yuan2014some}.

\NI In this paper, we introduce Nekrasov $Z$ tensor. We study the connection of Nekrasov $Z$ tensor and diagonally dominant tensor. We prove that the class of Nekrasov $Z$ tensor of even order with positive diagonal elements is a subclass of $P$-tensor.

The paper is organized as follows. Section 2 contains few essential and early observations. In Section 3, we develop Nekrasov $Z$ tensor.  We investigate its connection with the $P$-tensor.

\section{Preliminaries}
Here we consider some basic notations and present few useful definitions and theorems used in this paper. $x\in \mathbb{C}^n$ denotes a column vector and row transpose of $x$ is denoted by $x^T.$ A diagonal matrix $W=[w_{ij}]_{n \times n}=diag(w_1, \; ..., \; w_n)$ is defined as $w_{ij}=\left \{ \begin{array}{ll}
	  w_i  &;\; \forall \; i=j, \\
	  0  &; \; \forall \; i \neq j.
	   \end{array}  \right.$

\begin{defn}\cite{mangasarian1976linear}
$A\in \mathbb{R}^{[2,n]},$ is a $Z$-matrix if all its off-diagonal elements are nonpositive.
\end{defn}	   

\begin{defn}\cite{cvetkovic2009new,bailey1969bounds,kolotilina2015some, garcia2014error}
A matrix $M=(m_{ij})\in \mathbb{C}^{[2,n]},$ is a Nekrasov matrix if $\lvert m_{ii} \rvert> \Lambda_i(M), \; \forall\; i\in [n],$ where
\begin{center}
    $\Lambda_i(M) =\left\{\begin{array}{ll}
        \sum_{j=2}^n \lvert m_{ij} \rvert &, i=1,  \\
        \sum_{j=1}^{i-1} \lvert m_{ij} \rvert \frac{\Lambda_j(M)}{\lvert m_{jj} \rvert} + \sum_{j=i+1}^n \lvert m_{ij} \rvert &, i=2,3,...,n.
    \end{array} \right.$
\end{center}
\end{defn}

\begin{defn}\cite{orera2019accurate}
$A\in \mathbb{R}^{[2,n]}$ is Nekrasov $Z$-matrix, if $A$ is a Nekrasov mtrix as well as a $Z$-matrix.
\end{defn}


\noindent Let $\mathcal{O}$ be the zero tensor with zero at each entry of $\mathcal{O}.$ For $\mathcal{T}\in \mathbb{C}^{[m,n]}$ and $x\in \mathbb{C}^n,\; \mathcal{T}x^{m-1}\in \mathbb{C}^n $ is a vector defined by
\[ (\mathcal{T}x^{m-1})_i = \sum_{i_2, ..., i_m =1}^{n} t_{i_1  ...i_m} x_{i_2}  \cdots x_{i_m} , \mbox{   for all } i \in [n] \]
$\mathcal{T}x^m\in \mathbb{C} $ is a scalar defined by
\[ x^T\mathcal{T}x^{m-1}= \mathcal{T}x^m = \sum_{i_1, ..., i_m =1}^{n} t_{i_1  ...i_m} x_{i_1} \cdots x_{i_m}. \]

\NI Let $\mathcal{U}$ and $\mathcal{V}$ be two $n$ dimensional tensor of order $p \geq 2$ and $r \geq 1,$ respectively. Shao \cite{shao2013general} introduced the general product of $\mathcal{U}$ and $\mathcal{V}.$ Let $\mathcal{U} \cdot \mathcal{V} = \mathcal{T}.$ Then $\mathcal{T}$ is an $n$ dimensional tensor of order $((p-1)(r-1)) + 1$ where the elements of $\mathcal{T}$ are given by
\[t_{j \beta_1 \cdots \beta_{p-1} } =\sum_{j_2, \cdots ,j_p \in [n]} u_{j j_2 \cdots j_p} v_{j_2 \beta_1} \cdots v_{j_p \beta_{p-1}},\] where $j \in [n]$, $\beta_1, \cdots, \beta_{p-1} \in [n]^{r-1}$



\begin{defn}\cite{song2016properties}
Given $\mathcal{T} \in \mathbb{R}^{[m,n]} $ and $q\in \mathbb{R}^n$, $x$ solves TCP$(\mathcal{T},q)$ if there exists a $x$ satisfying the equation (\ref{tensor comp equation}).
\end{defn}

\begin{defn}\cite{song2015properties}
$\mathcal{T}\in \mathbb{R}^{[m,n]} $ is a $P$-tensor, if for any $x\in \mathbb{R}^n \backslash \{0\}$, $\exists$ a $j\in [n]$ such that $x_j \neq 0$ and $x_j (\mathcal{T}x^{m-1})_j >0.$
\end{defn}

\NI The row subtensors are defined in \cite{shao2016some}. Here for the sake of convenience we denote $i$th rowsubtensor of $\mathcal{T}$ by $\mathcal{R}_i(\mathcal{T}).$
\begin{defn}\cite{shao2016some}
For each $i$ the $i$th row subtensor of $\mathcal{T}\in \mathbb{C}^{[m,n]}$ is denoted by $\mathcal{R}_i(\mathcal{T})$ and its elements are given as $(\mathcal{R}_i(\mathcal{T}))_{i_2 ... i_m}=(t_{i i_2... i_m})$, where $i_l\in [n]$ and $2\leq l \leq m.$
\end{defn}

\begin{defn}\cite{zhang2018nekrasov}
Let $\mathcal{T}\in \mathbb{C}^{[m,n]}$ such that $\mathcal{T}=(t_{i_1 i_2 ... i_m})$ and
\begin{equation*}
    R_i(\mathcal{T})= \sum_{(i_2 ... i_m) \neq (i ... i)} \lvert t_{i i_2 ... i_m} \rvert,\; \forall \; i.
\end{equation*}
The tensor $\mathcal{T}$ is diagonally dominant (strict diagonally dominant) if $\lvert t_{i ... i} \rvert \geq (>) R_i(\mathcal{T})$ for all $i\in [n].$
\end{defn}

\begin{defn}\cite{zhang2018nekrasov}
$\mathcal{T}\in \mathbb{C}^{[m,n]}$ is quasidiagonally dominant tensor if there exists a diagonal matrix $W=diag(w_1, w_2, \cdots, w_n)$ such that $\mathcal{T} W$ is strictly diagonally dominant tensor, i.e.,
\begin{equation*}
  \lvert t_{i ... i} \rvert w_i^{m-1} > \lvert t_{i i_2 \cdots i_m}\rvert w_{i_2} \cdots w_{i_m}, \;\;\; \forall\; i\in [n]. 
\end{equation*}
\end{defn}

\NI Qi \cite{qi2005eigenvalues} and Lim \cite{lim2005singular} proposed the idea of eigenvalues and eigenvectors in case of tensors. If a pair $(\lambda, x) \in \mathbb{C}\times (\mathbb{C}^n \backslash \ \{0\})$ satisfies the equation $\mathcal{T} x^{m-1} = \lambda x^{m-1},$ then $\lambda$ is an eigenvalue of $\mathcal{T}$, and $x$ is an eigenvector with respect to $\lambda$ of $\mathcal{T}$, for which $x^{[m−1]} = (x_1^{m−1},... , x_n^{m−1})^T.$ The spectral radius of $\mathcal{T}$ is defined as $\rho(\mathcal{T}) = \max\{|\lambda| \; : \; \lambda \text{ is an eigen value of } \mathcal{T}\}.$

\begin{defn}\cite{ding2013m}
$\mathcal{T}$ is a $Z$ tensor if all its off diagonal elements are nonpositive.
\end{defn}
\begin{defn}\cite{ding2013m}
A $Z$ tensor $\mathcal{T}$ is an $M$ tensor if $\exists$ $\mathcal{B}\; \geq \mathcal{O}$ and $s>0$ such that $\mathcal{T}= s \mathcal{I} - \mathcal{B}$, where $s\geq \rho(\mathcal{B})$. $\mathcal{T}$ is a nonsingular $M$ tensor if $s>\rho(\mathcal{B}).$
\end{defn}

\begin{defn}\cite{ding2013m}
For a tensor $\mathcal{T}= t_{i_1 i_2 ... i_m}\in \mathbb{C}^{[m,n]},$ $\mathcal{C}(\mathcal{T})= c_{i_1 i_2 ... i_m}$ is a comparison tensor of $\mathcal{T}$ if 
\begin{center}
    $c_{i_1 i_2 ... i_m}=\left\{\begin{array}{ll}
       |t_{i_1 i_2 ... i_m}|  & , \text{ if } (i_1, i_2, ... ,i_m) =(i,i, ..., i), \\
       -|t_{i_1 i_2 ... i_m}|  & , \text{ if } (i_1, i_2, ... ,i_m) \neq (i,i, ..., i).
    \end{array} \right.$
\end{center}
\end{defn}

\begin{defn}\cite{ding2013m}
$\mathcal{T}$ is an $H$ tensor, if its comparison tensor is an $M$-tensor, and it is called a nonsingular $H$ tensor if its comparison tensor is a nonsingular $M$-tensor.
\end{defn}

\begin{defn}\cite{zhang2018nekrasov}
For $\mathcal{T}=(t_{i_1 i_2 ... i_m}) \in \mathbb{C}^{[m,n]}, \; t_{i i ... i}\neq 0,$ $\forall \;i$ and
\begin{align*}
    \Lambda_1(\mathcal{T}) & =R_1(\mathcal{T}),\\ 
    \Lambda_i(\mathcal{T}) & =\sum_{i_2...i_m \in [i-1]^{m-1}} \lvert t_{i i_2 ... i_m} \rvert \left(\frac{\Lambda_{i_2} (\mathcal{T})}{\lvert t_{i_2  ... i_2} \rvert}\right)^{\frac{1}{m-1}} \cdots \left(\frac{\Lambda_{i_m} (\mathcal{T})}{\lvert t_{i_m  ... i_m} \rvert}\right)^{\frac{1}{m-1}}\\ 
    & \qquad + \sum_{i_2...i_m \notin [i-1]^{m-1}, (i_2 ... i_m)\neq (i ... i)} \lvert t_{i i_2 ... i_m} \rvert, \;\;\;\; i=2,3,...,n. 
\end{align*}
$\mathcal{T}=(t_{i_1 i_2 ... i_m})\in \mathbb{R}^{[m,n]}$ is a Nekrasov tensor if 
 \begin{equation}
     \lvert t_{i  ... i} \rvert > \Lambda_i(\mathcal{T}), \; \forall\; i\in [n].
 \end{equation}
\end{defn}

\begin{theorem}\cite{ding2013m}\label{nonsingular H if and only if quasi-SDD}
$\mathcal{T}$ is a nonsingular $H$ tensor iff it is a quasi-strictly diagonally dominant tensor.
\end{theorem}

\begin{theorem}\cite{zhang2018nekrasov}\label{nekrasov implies nonsingular H}
If a tensor $\mathcal{T}=(t_{i_1 i_2 ... i_m})\in \mathbb{C}^{[m,n]}$ is a Nekrasov tensor, then $\mathcal{T}$ is a nonsingular $H$ tensor.
\end{theorem}

\begin{theorem}\cite{ding2018p}\label{nonsingular H implies P}
A nonsingular $H$ tensor of even order with all positive diagonal elements is a $P$-tensor.
\end{theorem}

\begin{theorem}\cite{bai2016global}
For $q\in \mathbb{R}^n$ and a $P$-tensor $\mathcal{T} \in \mathbb{R}^{[m,n]}$, the solution set of TCP$(\mathcal{T},q)$ is nonempty and compact.
\end{theorem}

\section{Main results}

We begin by introducing Nekrasov $Z$ tensor.

\begin{defn}
$\mathcal{T}\in \mathbb{R}^{[m,n]}$ is a Nekrasov $Z$ tensor if $\mathcal{T}$ is a Nekrasov tensor as well as a $Z$ tensor.
\end{defn}

The following is an example of a Nekrasov $Z$ tensor.

\begin{examp}\label{example of nekrasov Z tensor}
We consider the Example 3.3 of \cite{zhang2018nekrasov}. Let $\mathcal{B}\in \mathbb{R}^{[4,4]}$ be such that $b_{1111}=8,\; b_{2222}=3.8,\; b_{3333}=3,\; b_{4444}=10,\; b_{1112}=b_{2111}=b_{1211}=b_{1121}=-1,\;$ $b_{3222}=b_{2322}=b_{2232}=b_{2223}=-1,$ $b_{4441}=b_{4414}=b_{4144}=b_{1444}=-3$ and all other elements of $\mathcal{B}$ are zeros. Then $R_1(\mathcal{B})=6,\; R_2(\mathcal{B})=4,\; R_3(\mathcal{B})=1,\; R_4(\mathcal{B})=9.$ Since $|b_{2222}=| < R_2(\mathcal{B}),$ so $\mathcal{B}$ is not diagonally dominant tensor. However $\lambda_1(\mathcal{B})=6,\; \lambda_2(\mathcal{B})=3.75,\;\lambda_3(\mathcal{B})=0.98,\;\lambda_4(\mathcal{B})=9.$ Also $|b_{ii...i}|>\lambda_i(\mathcal{B}),\; \forall\; i\in [4].$ Hence $\mathcal{B}$ is Nekrasov tensor. Also all the off-diagonal elements of $\mathcal{B}$ are nonpositive. Therefore $\mathcal{B}$ is a $Z$ tensor. Hence $\mathcal{B}$ is a Nekrasov $Z$ tensor.
\end{examp}

Now we show that $P$-tensor contains even order Nekrasov $Z$ tensor with positive diagonal elements.

\begin{theorem}\label{1st theorem of Nekrasov Z tensor}
A Nekrasov $Z$ tensor of even order with positive diagonal elements is a $P$-tensor.
\end{theorem}
\begin{proof}
Let $\mathcal{T}$ be a Nekrasov $Z$ tensor which has positive diagonal elements. By Theorem \ref{nekrasov implies nonsingular H}, a Nekrasv tensor is a $H$ tensor which is also nonsingular. Therefore $\mathcal{T}$ is a nonsingular $H$ tensor which has positive diagonal elements. Therefore by Theorem \ref{nonsingular H implies P} $\mathcal{T}$ is a $P$-tensor. 
\end{proof}

Since a Nekrasov tensor is a nonsingular $H$ tensor, then there exists a positive diagonal matrix $W$ such that $\mathcal{T}$ is a quasi-diagonally dominant tensor. Here we are developing a diagonal matrix $W$ for which $\mathcal{T}W$ is a diagonally dominant tensor.

\begin{theorem}
Suppose $\mathcal{T}\in \mathbb{R}^{[m,n]},$ $m$ is even. Let $\mathcal{T}$ be a Nekrasov $Z$ tensor whose diagonal elements are positive and let $W$ be a diagonal matrix such that,
\begin{equation}
    W= diag\left( \left(\frac{\Lambda_1(\mathcal{T})}{t_{11\cdots 1}}\right)^{\frac{1}{m-1}},\;  \left(\frac{\Lambda_2(\mathcal{T})}{t_{22\cdots 2}}\right)^{\frac{1}{m-1}},\; \cdots\; , \left(\frac{\Lambda_n(\mathcal{T})}{t_{nn\cdots n}}\right)^{\frac{1}{m-1}} \right)
\end{equation}
Then $\mathcal{T}W$ is a diagonally dominant $Z$ tensor.
\end{theorem}
\begin{proof}
Firstly note that $\frac{\Lambda_i(\mathcal{T})}{t_{ii\cdots i}} \geq 0$ and so $\left(\frac{\Lambda_i(\mathcal{T})}{t_{ii\cdots i}}\right)^{\frac{1}{m-1}} \geq 0, \; \forall\; i\in [n],$ as $\mathcal{T}$ is a Nekrasov $Z$ tensor with positive diagonal elements and $m$ is even. So $W$ is a nonnegative diagonal matrix. Let $\mathcal{B}=(b_{i_1 i_2 ...i_m})\in \mathbb{R}^{[m,n]}$ be such that $\mathcal{B} = \mathcal{T}W.$ Then
\begin{equation}
    b_{i_1 i_2 ... i_m}= t_{i_1 i_2 ... i_m} \left(\frac{\Lambda_{i_2}(\mathcal{T})}{t_{{i_2}{i_2}\cdots {i_2}}} \right)^{\frac{1}{m-1}} \cdots \left(\frac{\Lambda_{i_m}(\mathcal{T})}{t_{{i_m}{i_m}\cdots {i_m}}} \right)^{\frac{1}{m-1}}
\end{equation}
and
\begin{equation}
    b_{ii ...i} =\Lambda_i(\mathcal{T}).
\end{equation}
It is easy to observe that the signs of $\mathcal{B}$ are same as the signs of $\mathcal{T}.$ This infers that $\mathcal{B}$ is a $Z$ tensor, since $\mathcal{T}$ is a $Z$ tensor. Now we show that $\mathcal{B}$ is a  diagonally dominant tensor. We write
\begin{equation}\label{for the conclusion of Nekrasov Z}
   t_{k\cdots k}= \lvert t_{k\cdots k} \rvert > \Lambda_k(\mathcal{T}), \;\; \forall\; k\in [n],
\end{equation}
since $\mathcal{T}$ is a Nekrasov tensor.
Now for each $i\in [n]$ we obtain,
\begin{align*}
  |b_{i\cdots i}| &=b_{i\cdots i}\\
                  &=\Lambda_i (\mathcal{T})\\
                  &=\sum_{i_2 ... i_m \in [i-1]^{m-1}} \lvert t_{i i_2 ... i_m} \rvert \left(\frac{\Lambda_{i_2} (\mathcal{T})}{\lvert t_{i_2  ... i_2} \rvert}\right)^{\frac{1}{m-1}} \cdots \left(\frac{\Lambda_{i_m} (\mathcal{T})}{\lvert t_{i_m  ... i_m} \rvert}\right)^{\frac{1}{m-1}}\\
                  & \qquad + \sum_{i_2 ... i_m \notin [i-1]^{m-1}, (i_2 ... i_m)\neq (i ... i)} \lvert t_{i i_2 ... i_m} \rvert\\
                  &\geq \sum_{i_2 ... i_m \in [i-1]^{m-1}} \lvert t_{i i_2 ... i_m} \rvert \left(\frac{\Lambda_{i_2} (\mathcal{T})}{\lvert t_{i_2  ... i_2} \rvert}\right)^{\frac{1}{m-1}} \cdots \left(\frac{\Lambda_{i_m} (\mathcal{T})}{\lvert t_{i_m  ... i_m} \rvert}\right)^{\frac{1}{m-1}} \\
                  & \qquad + \sum_{i_2 ... i_m \notin [i-1]^{m-1}, (i_2 ... i_m)\neq (i ... i)} \lvert t_{i i_2 ... i_m} \rvert \left(\frac{\Lambda_{i_2} (\mathcal{T})}{\lvert t_{i_2  ... i_2} \rvert}\right)^{\frac{1}{m-1}} \cdots \left(\frac{\Lambda_{i_m} (\mathcal{T})}{\lvert t_{i_m  ... i_m} \rvert}\right)^{\frac{1}{m-1}} \text{ by (\ref{for the conclusion of Nekrasov Z})} \\
                  &=\sum_{i_2 ... i_m \in [n]^{m-1}, (i_2 ... i_m)\neq (i ... i) } \lvert t_{i i_2 ... i_m} \rvert \left(\frac{\Lambda_{i_2} (\mathcal{T})}{\lvert t_{i_2  ... i_2} \rvert}\right)^{\frac{1}{m-1}} \cdots \left(\frac{\Lambda_{i_m} (\mathcal{T})}{\lvert t_{i_m  ... i_m} \rvert}\right)^{\frac{1}{m-1}}\\
                  &=\sum_{(i_2 ... i_m)\neq (i ... i) } |b_{i i_2 ... i_m}|.
\end{align*}
This implies that $\mathcal{B}$ is a diagonally dominant tensor.
\end{proof}

Now we want to find a class of tensors that will contain Nekrasov $Z$ tensors with positive diagonal elements. To do so we use a decomposition of tensors which will be useful in our discussion.\\

\NI Given a tensor $\mathcal{T}=(t_{i_1 i_2 ... i_m})\in \mathbb{R}^{[m,n]},$ we can write $\mathcal{T}$ as
\begin{equation}\label{decomposition equation}
    \mathcal{T} = \mathcal{B}^+ + \mathcal{C},
\end{equation}
where $\forall \; i\in [n]$ the elements of rowsubtensors of $\mathcal{B}$ and $\mathcal{C}$ are given by,
\begin{equation}\label{law for B+}
(R_i(\mathcal{B}^+))_{i_2 \cdots i_m} = t_{i i_2 ... i_m} - r_i^+, \;\forall \; i_l\in [n], \; l\in [m]
\end{equation}
and
\begin{equation}\label{law for C}
(R_i(\mathcal{C}))_{i_2 \cdots i_m} =r_i^+ , \;\forall \; i_l\in [n], \; l\in [m]..
\end{equation}
with
\begin{equation}\label{law for r_i^+}
    r_i^+ =\max_{i_2, ..., i_m \in [n], \; ( i_2 \cdots i_m) \neq (i \cdots i)} \{ 0, t_{i i_2 ... i_m} \}
\end{equation}

\begin{prop}
Let $\mathcal{T}=(t_{i_1 i_2 ... i_m})\in \mathbb{R}^{[m,n]}.$ If $\mathcal{T}$ is decomposed as given in the equation (\ref{decomposition equation}), then $\mathcal{B}^+$ is a $Z$ tensor and $\mathcal{C}$ is nonnegative tensor.
\end{prop}
\begin{proof}
Firstly, by the definition of the tensor $\mathcal{B}^+$ we have $b_{i_1 i_2 ... i_m} = t_{i_1 i_2 ... i_m} - r_{i_1}^+,$ for all $i_1,i_2, ... ,i_m \in [n].$
Now from the definition of $r_{i_1}^+$ given by equation (\ref{law for r_i^+}) for all $i_1, i_2, ... ,i_m \in [n],$ we write
\begin{align*}
    b_{i_1 i_2 ... i_m} &= t_{i_1 i_2 ... i_m} - r_{i_1}^+\\
    &= t_{i_1 i_2 ... i_m} - \max_{i_2, ..., i_m \in [n], \; ( i_2 \cdots i_m) \neq (i \cdots i)} \{ 0, t_{i i_2 ... i_m} \} \\
    &\leq 0, \text{ when } ( i_2 \cdots i_m) \neq (i \cdots i).
\end{align*}
Therefore all off diagonal elements of the tensor $\mathcal{B}^+$ are nonpositive. Hence the tensor $\mathcal{B}^+$ is a $Z$ tensor.

Secondly, observe that the elements of $\mathcal{C}$ are $r_i^+,$ $\forall \;i\in[n].$ Now by equation (\ref{law for r_i^+}) we write
\begin{equation*}
     r_i^+ =\max_{i_2, ..., i_m \in [n], \; ( i_2 \cdots i_m) \neq (i \cdots i)} \{ 0, t_{i i_2 ... i_m} \}\geq 0.
\end{equation*}
 This implies $\mathcal{C}$ is a nonnegative tensor.
\end{proof}

\begin{remk}
If the tensor $\mathcal{B}^+$ is a Nekrasov tensor then $\mathcal{B}^+$ becomes a Nekrasov $Z$ tensor.
\end{remk}

\begin{remk}
The class of tensors $\mathcal{T}$ that can be expressed as $\mathcal{T} = \mathcal{B}^+ + \mathcal{C},$ where $\mathcal{B}^+$ and $\mathcal{C}$ are constructed by (\ref{law for B+}), (\ref{law for C}) and (\ref{law for r_i^+}) with $\mathcal{B}^+$ as a Nekrasov tensor contains the class of Nekrasov $Z$ tensors.
\end{remk}

\begin{remk}
The $P$ tensors play a crucial role in tensor complementarity theory. We know that the SOL$(\mathcal{T}, q)$ is nonempty and compact if $\mathcal{T}$ is a $P$ tensor. Here we show that $\mathcal{T},$ a Nekrasov $Z$ tensor of even order which has positive diagonal elements, is a $P$-tensor. Hence we conclude that the solution set of TCP$(\mathcal{T}, q)$ is nonempty and compact.
\end{remk}

\section{Conclusion}
In this article, we introduce Nekrasov $Z$ tensor. For a Nekrasov $Z$ tensor $\mathcal{T}$ we find $W$ such that $\mathcal{T}W$ is a diagonally dominant tensor. We prove that the Nekrasov $Z$ tensors of even order with positive diagonal elements are contained in the class of $P$-tensor.

\section{Acknowledgment}
The author R. Deb acknowledges CSIR, India, JRF scheme for financial support.

\bibliographystyle{plain}
\bibliography{referencesall}

\end{document}